\documentclass[11pt, a4paper]{article}

\usepackage[ansinew]{inputenc}
\usepackage[english]{babel}
\usepackage{amsmath, mathrsfs, amsthm, amssymb, color, stmaryrd}
\usepackage{a4wide}

\usepackage{hyperref}

\newtheorem{theorem}{Theorem}[section]
\newtheorem{corollary}[theorem]{Corollary}
\newtheorem{definition}[theorem]{Definition}
\newtheorem{lemma}[theorem]{Lemma}
\newtheorem{remark}[theorem]{Remark}
\newtheorem{proposition}[theorem]{Proposition}
\newtheorem{example}[theorem]{Example}
\newtheorem{property}[theorem]{Property}
\newtheorem*{conjecture}{Conjecture}

\def\ba#1\ea{\begin{align*}#1\end{align*}}

\def \follmer {F\"{o}llmer}
\def \ito {It\^{o} }

\newcommand{\R}{\mathbb{R}}
\newcommand{\E}{\mathbb{E}}
\newcommand{\N}{\mathbb{N}}
\newcommand{\1}{\mathbf{1}}
\newcommand{\Z}{\mathbb{Z}}

\newcommand{\bb}[1]{\mathbb{#1}}
\newcommand{\mc}[1]{\mathcal{#1}}

\newcommand{\assign}{:=}
\newcommand{\dd}{\mathrm{d}}
\newcommand{\op}[1]{\ensuremath{\operatorname{#1}}}
\newcommand{\lb}{\llparenthesis}
\newcommand{\rb}{\rrbracket}

\begin{document}

\title{Pathwise integration and change of variable formulas \\
for continuous paths with arbitrary regularity}

\author{
  Rama Cont \footnote{LPSM, CNRS-Sorbonne Universit\'e \& Department of Mathematics,  Imperial College London. }
  \and
  Nicolas Perkowski \footnote{Institut f\"ur Mathematik,  Humboldt--Universit\"at zu Berlin}.
}

\date{March 2018.}
\maketitle

\begin{abstract}
We construct a pathwise integration theory, associated with a change of variable formula, for smooth functionals of continuous paths with arbitrary regularity defined in terms of the notion of  $p$-th variation along a sequence of time partitions. 
For  paths with finite $p$-th variation along a sequence of time partitions, we derive a change of variable formula for $p$ times continuously differentiable functions and show pointwise convergence of appropriately defined compensated Riemann sums.  

Results for functions are extended to regular path-dependent  functionals using the concept of vertical derivative of a functional. We show that the pathwise integral satisfies an `isometry' formula  in terms of $p$-th order variation   and obtain  a `signal plus noise' decomposition  for regular functionals of paths with strictly increasing $p$-th variation.
For less regular  ($C^{p-1}$) functions  we obtain a Tanaka-type change of variable formula using an appropriately defined  notion of  local time. 

These results extend to multidimensional paths and yield a natural   higher-order  extension of the concept of `reduced rough path'. We show that, while our  integral coincides with a rough-path integral for a certain rough path,  its construction is canonical and does not involve the specification of any rough-path superstructure. 
\end{abstract}



\section*{Introduction}

In his seminal paper {\it Calcul d'\ito sans probabilit\'es} \cite{follmer1981}, Hans \follmer{}  provided a pathwise proof of the \ito formula, using the concept of {\it quadratic variation along a sequence of partitions}, defined as follows.
A  path $S\in C([0,T],\mathbb{R})$ is said to have finite
quadratic variation along the sequence of partitions
$\pi_n=(0=t^{n}_0<t^n_1<\cdots<t^{n}_{N(\pi_n)}=T)$  if for any $t\in [0,T]$, the sequence of measures
\[
		\mu^n \assign \sum_{[t^n_j, t^n_{j + 1}] \in \pi_n} \delta (\cdot - t_j) |S(t^n_{j + 1}) - S(t^n_j)|^2
	\]
	converges weakly to a measure $\mu$ without atoms. The  continuous increasing function $[S] \colon  [0,T]\to \mathbb{R}_+$ defined by $[S](t)=\mu([0,t])$ is then called the quadratic variation of $S$ along $\pi$. 
	Extending this definition to vector-valued paths 
 F\"ollmer \cite{follmer1981} showed that, for integrands of the form  $\nabla f( S(t))$ with $f\in C^2(\mathbb{R}^d)$, one may define a pathwise integral $\int \nabla f(S(t))dS$  as a pointwise limit of  Riemann sums  along the sequence of partitions $(\pi_n)$
 and he obtained an \ito (change of variable) formula for $f( S(t))$ in terms of this pathwise integral: for $f\in C^2(\mathbb{R}^d), t\in [0,T]$,
 \ba  &&f(S(t))= \int_0^t \langle \nabla f(S(s)), \dd S(s)\rangle + \frac{1}{2}\int_0^t \langle \nabla^2 f(S(s)), \dd[S](s)\rangle,  \nonumber\\
 {\rm where}&& \int_0^t \langle \nabla f(S(s)) , \dd  S(s) \rangle:=\mathop{\lim}_{n\to\infty}\sum_{[t^n_j, t^n_{j + 1}] \in \pi_n} \langle \nabla f(S(t)), ( S(t^n_{j + 1} \wedge t) - S(t^n_j \wedge t))\rangle .
 \nonumber\ea
 This result  has many interesting ramifications and applications in the pathwise approach to stochastic analysis, and has been  extended in different ways, to less regular functions using the notion of pathwise local time \cite{bertoin1987,davis2018,perkowski2015}, as well as to path-dependent functionals and   integrands \cite{ananova2017,cont2012,CF10B, perkowski2016}.
 
The central role played by the concept of  quadratic variation has led to the assumption that they do not extend to  less regular paths with infinite quadratic variation. Integration theory and change of variables formulas for processes with infinite quadratic variation, such as fractional Brownian motion and other fractional processes, have relied on probabilistic, rather than pathwise constructions \cite{carmona2003,coutin2007,gradinaru2003}. Furthermore, the change of variable formulae obtained using these methods are valid for a restricted range of Hurst exponents (see \cite{nualart2006} for an overview).

In this work, we show that  \follmer's pathwise \ito calculus may be extended to paths with arbitrary regularity, in a strictly pathwise setting, using the concept of  $p$-th variation along a sequence of time partitions.
For  paths with finite $p$-th variation along a sequence of time partitions, we derive a change of variable formula for $p$ times continuously differentiable functions and show pointwise convergence of appropriately defined compensated Riemann sums. 
This result may be seen as the natural extension of the results of \follmer~\cite{follmer1981}  to paths of lower regularity.
Our results apply in particular to paths of fractional Brownian motions with arbitrary Hurst exponent, and yield pathwise proofs for results previously derived using probabilistic methods, without any restrictions on the Hurst exponent.

Using the concept of the vertical derivative of a functional \cite{CF10B}, we extend these results to regular path-dependent  functionals of such paths.
We obtain an `isometry' formula in terms of $p$-th order variations for the pathwise integral and a `signal plus noise' decomposition for regular functionals of paths with strictly increasing $p$-th variation, extending the results of \cite{ananova2017} obtained for the case  $p=2$ to arbitrary even integers $p\geq 2$.

The extension to less regular (i.e. not $p$ times differentiable) functions is more delicate and requires defining an appropriate higher-order analogue of semimartingale local time, which we introduce through an appropriate spatial localization of the $p$-th order variation. Using this higher-order concept of local time, we obtain a Tanaka-type change of variable formula for less regular (i.e. $p-1$ times differentiable) functions. We conjecture that these results apply in particular to paths of fractional Brownian motion and other fractional processes.

Finally, we consider extensions of these results to multidimensional paths and link them with rough path theory; the corresponding concepts yield a natural  higher order extension to the concept of `reduced rough path' introduced by Friz and Hairer~\cite[Chapter~5]{FrizHairer}.

\paragraph{Outline}
Section \ref{sec.pvariation} introduces the notion of $p$-th variation along a sequence of partitions and derives a change of variable formula for $p$ times continuously differentiable functions of  paths with finite $p$-th variation (Theorem \ref{thm:follmer-ito}).
An extension of these results to path-dependent functionals is discussed in Section \ref{sec.pathdependent}: Theorem \ref{thm.functional} gives a functional change of variable formula for regular functionals of paths with finite $p$-th variation.

Section \ref{sec.isometry} studies the corresponding pathwise integral in more detail. We first show (Theorem \ref{thm.isometry}) that the integral exhibits an `isometry' property in terms of the $p$-th order variation and use this property to obtain a unique `signal plus noise' decomposition where the components are discriminated in terms of their $p$-th order variation (Theorem \ref{thm.decomposition}).

The extension of these concepts to multidimensional paths and the relation to the concept of `reduced rough paths' are discussed in Section \ref{sec.multidimensional}.

\paragraph{Acknowledgement}
This work was completed while N.P. was visiting the University of Technology Sydney as Bruti-Liberati Visiting Fellow. N.P. is grateful for the kind hospitality at UTS and for the generous financial support through the Bruti-Liberati Scholarship. N.P. also gratefully acknowledges financial support by the DFG via Research Unit FOR 2402.

\section{Pathwise calculus for paths with finite $p$-th variation}\label{sec.pvariation}

\subsection{$p$-th variation along a sequence of partitions}

We introduce, in the spirit of \follmer ~\cite{follmer1981}, the concept of $p$-th variation along a sequence of partitions $\pi_n=\{t_0^n, \dots, t^n_{N(\pi_n)}\}$ with $t_0^n=0<...< t^n_k<...< t^n_{N(\pi_n)}=T$. Define the \emph{oscillation} of $S \in C([0,T],\R)$ along $\pi_n$ as
\[
	\op{osc}(S,\pi_n) := \max_{[t_j,t_{j+1}] \in \pi_n} \max_{r,s \in [t_j,t_{j+1}]} |S(s) - S(r)|.
\]
Here and in the following we write $[t_j,t_{j+1}] \in \pi_n$ to indicate that $t_j$ and $t_{j+1}$ are both in $\pi_n$ and are immediate successors (i.e. $t_j < t_{j+1}$ and $ \pi_n \cap(t_j , t_{j+1}) = \emptyset$).

\begin{definition}[$p$-th variation along a sequence of partitions]\label{def:p-var} Let $p>0$. A  continuous path $S \in C([0,T],\R)$ is said to have a $p$-th variation along a sequence of
partitions $\pi=(\pi_n)_{n\geq 1}$ if $\op{osc}(S,\pi_n)\to 0$ and the sequence of measures
	\[
		\mu^n \assign \sum_{[t_j, t_{j + 1}] \in \pi_n} \delta (\cdot - t_j) |S(t_{j + 1}) - S(t_j)|^p
	\]
	converges weakly to a measure $\mu$ without atoms. In that case we write $S \in V_p(\pi)$ and $[S]^p(t) := \mu([0,t])$ for $t \in [0,T]$, and we call $[S]^p$ the \emph{$p$-th variation} of $S$.
%
\end{definition}

\begin{remark}\label{rem.pvariation}
\begin{enumerate}
	\item Functions in $V_p(\pi)$ do not necessarily have finite $p$-variation in the usual sense. Recall that the $p$-variation of a function $f \in C([0,T],\R)$ is defined as \cite{dudley2011}
		\[
			\|f \|_{\op{p-var}} \assign \Big(\sup_{\pi\in \Pi([0,T])} \sum_{[t_j, t_{j + 1}] \in \pi} |f(t_{j + 1}) - f(t_j)|^p\Big)^{1/p},
		\]
		where the supremum is taken over the set $\Pi([0,T])$ of all partitions $\pi$ of $[0,T]$. A typical example is the Brownian motion $B$, which has quadratic variation $[B]^2(t) = t$ along any refining sequence of partitions almost surely while at the same time having infinite 2-variation almost surely \cite{dudley2011,taylor1972}: $$\mathbb{P}\left( \|B\|_{\op{2-var}} = \infty\right) =1.$$  
	\item If $S \in V_p(\pi)$ and $q > p$, then $S \in V_q(\pi_n)$ with $[S]^q \equiv 0$.
\end{enumerate}
\end{remark}
The following lemma gives a simple characterization of this property:
\begin{lemma} Let $S \in C([0,T],\R)$.
 $S \in V_p(\pi)$ if and only if there exists a continuous function $[S]^p$ such that
	\begin{equation}
	\forall t\in[0,T],\qquad \sum_{\substack{[t_j, t_{j + 1}] \in \pi_n: \\ t_j \le t}} |S(t_{j + 1}) - S(t_j)|^p\mathop{\longrightarrow}^{n\to\infty} [S]^p(t).\label{eq.pointwisecv}\end{equation}
 If this property holds, then the convergence in \eqref{eq.pointwisecv} is uniform.
\end{lemma}
 Indeed, the weak convergence of measures on $[0,T]$ is equivalent to the pointwise convergence of their cumulative distribution functions at all continuity points of the limiting cumulative distribution function, and if the limiting cumulative distribution function is continuous, the convergence is  uniform.
 
\begin{example}\label{ex:fBm}
If $B$ is a fractional Brownian motion with Hurst index $H \in (0,1)$ and $\pi_n = \{kT/n: k \in \N_0\} \cap[0,T]$, then $B \in V_{1/H}(\pi)$ and $[B]^{1/H}(t) = t \E[|B_1|^{1/H}]$, see~\cite{pratelli2006,Rogers1997}.
\end{example}

\subsection{Pathwise integral and change of variable formula}
A key observation of \follmer~\cite{follmer1981} was that, for $p=2$, Definition~\ref{def:p-var} is sufficient to obtain  a pathwise It\^o formula for ($C^2$) functions of $S \in V_2(\pi_n)$. We will show that in fact \follmer's argument may be applied for any even integer $p$:

\begin{theorem}[Change of variable formula for paths with finite $p$-th variation]\label{thm:follmer-ito}
	Let $p \in \N$ be even, let $(\pi_n)$ be a given sequence of partitions, and let $S \in V_p(\pi)$. Then for every $f \in C^p(\R,\R)$ the pathwise change of variable formula
	\[
		f(S(t)) - f(S(0)) = \int_0^t f'(S(s))\dd S(s) + \frac{1}{p!} \int_0^t f^{(p)}(S(s)) \dd [S]^p(s),
	\]
	holds, where the integral 
	\[
		\int_0^t f'(S(s))\dd S(s) \assign \lim_{n \rightarrow \infty} \sum_{[t_j, t_{j + 1}] \in \pi_n} \sum_{k=1}^{p-1} \frac{f^{(k)} (S(t_j))}{k!}
   (S(t_{j + 1} \wedge t) - S(t_j \wedge t))^k
	\]
	is defined as a (pointwise) limit of compensated Riemann sums.
\end{theorem}
\begin{proof}
	Applying a Taylor expansion at order $p$ to the increments of $f(S)$ along the partition,  we obtain
	\begin{align}\label{eq:ito-pr1}
  		f (S(t)) - f (S(0)) & = \sum_{[t_j, t_{j + 1}] \in \pi_n} (f (S(t_{j + 1}\wedge t)) - f (S(t_j \wedge t)))\\ \nonumber
  		& = \sum_{[t_j, t_{j + 1}] \in \pi_n} \sum_{k=1}^p \frac{f^{(k)} (S(t_j))}{k!} (S(t_{j + 1}\wedge t) - S(t_j \wedge t))^k\\ \nonumber
  		&\quad + \sum_{[t_j, t_{j + 1}] \in \pi_n} \int_0^1 \dd \lambda \frac{(1 - \lambda)^{p - 1}}{(p - 1) !} (S(t_{j + 1} \wedge t) - S(t_j \wedge t))^p \\ \nonumber
  &\hspace{80pt}\times \big(f^{(p)} (S(t_j) + \lambda (S(t_{j + 1} \wedge t) - S(t_j \wedge t))) - f^{(p)} (S(t_j))\big) .
	\end{align}
	Since the image of $(S(t))_{t \in [0,T]}$ is compact, we may assume without loss of generality that $f$ is compactly supported; then the remainder on the right hand side is bounded by
	\ba
		&\Big|\sum_{[t_j, t_{j + 1}] \in \pi_n} \int_0^1 \dd \lambda \frac{(1 - \lambda)^{p - 1}}{(p - 1) !} (S(t_{j + 1} \wedge t) - S(t_j \wedge t))^p \\
		&\hspace{100pt} \times \big(f^{(p)} (S(t_j) + \lambda (S(t_{j + 1} \wedge t) - S(t_j \wedge t))) - f^{(p)}(S(t_j))\big) \Big|\\
		&\hspace{50pt} \le C(f, S, \pi_n, p) \mu_n([0,t]) 
	\ea
	with a constant $C(f, S, \pi_n, p) > 0$ that converges to zero for $n \rightarrow \infty$, and therefore the remainder vanishes for $n \to \infty$.
	Since $S \in V_p(\pi)$ we know that
	\[
		\lim_{n\to \infty} \sum_{[t_j, t_{j + 1}] \in \pi_n} \frac{f^{(p)} (S(t_j))}{p!} (S(t_{j + 1}\wedge t) - S(t_j \wedge t))^p = \frac{1}{p!} \int_0^t f^{(p)}(S(s)) \dd [S]^p(s),
	\]
	and therefore we obtain from~\eqref{eq:ito-pr1}
	\ba
		&\lim_{n \to \infty} \sum_{[t_j, t_{j + 1}] \in \pi_n} \sum_{k=1}^{p-1} \frac{f^{(k)} (S(t_j))}{k!} (S(t_{j + 1} \wedge t) - S(t_j \wedge t))^k \\
		&\hspace{50pt} = f(S(t)) - f(S(0)) - \frac{1}{p!} \int_0^t f^{(p)}(S(s)) \dd [S]^p(s),
	\ea
	and we simply define $\int_0^t f'(S(s)) \dd S(s)$ as the limit on the left hand side.
\end{proof}

\begin{remark}[Relation with Young integration and rough-path integration]{\em The expression
\[
		\sum_{[t_j, t_{j + 1}] \in \pi_n} \sum_{k=1}^{p-1} \frac{f^{(k)} (S(t_j))}{k!}
   (S(t_{j} \wedge t) - S(t_j \wedge t))^k
	\]
	is a `compensated Riemann sum'. 
	Note however that, given the assumptions on $S$, the pathwise integral appearing in the formula cannot be defined as a Young integral, {\it even} after substracting the compensating terms. This relates to the observation in Remark \ref{rem.pvariation} that p-variation can be infinite for $S\in V_p(\pi)$.
	
	When $p=2$ it  reduces to an ordinary (left) Riemann sum. For $p>2$ such  compensated Riemann sums appear in the construction of `rough path integrals' ~\cite{FrizHairer,gubinelli2004}. Let $X \in C^\alpha([0,T],\R)$ be $\alpha$-H\"older continuous for some $\alpha \in (0,1)$, and write $q = \lfloor \alpha^{-1} \rfloor$. We can enhance $X$ uniquely into a (weakly) geometric rough path $(\bb X^1_{s,t}, \bb X^2_{s,t}, \dots, \bb X^q_{s,t})_{0 \le s \le t \le T}$, where $\bb X^k_{s,t} \assign  (X(t) - X(s))^k/k!$. Moreover, for $g \in C^{q+1}(\R,\R)$ the function $g'(X)$ is controlled by $X$ with Gubinelli derivatives
	\begin{align*}
		g'(X(t)) - g'(X(s)) & = \sum_{k=1}^{q-1} \frac{g^{(k+1)}(X(s))}{k!} (X(t) - X(s))^k + O(|t-s|^{q \alpha}) \\
		& = \sum_{k=1}^{q-1} g^{(k+1)}(X(s)) \bb X^k_{s,t} + O(|t-s|^{q \alpha}),
	\end{align*}
	and therefore the controlled rough path integral $\int_0^t g'(X(s)) \dd X(s)$ is given by
	\[
		\lim_{|\pi|\to 0} \sum_{[t_j, t_{j+1}] \in \pi} \sum_{k=1}^q g^{(k)}(X(s)) \bb X^k_{s,t} = \lim_{|\pi|\to 0} \sum_{[t_j, t_{j+1}] \in \pi} \sum_{k=1}^q g^{(k)}(X(s)) \frac{(X(t) - X(s))^k}{k!},
	\]
	where $|\pi|$ denotes the mesh size of the partition $\pi$, and which is exactly the type of compensated Riemann sum that we used to define our integral. The link between our approach and rough path integration is explained in more detail in Section~\ref{sec.multidimensional-rp} below.}
\end{remark}

\begin{remark}{\em
	In principle we could apply similar arguments for odd integers $p$ if instead of $S \in V_p(\pi)$ we assumed that $\sum_{[t_j, t_{j + 1}] \in \pi_n} \delta (\cdot - t_j) (S(t_{j + 1}) - S(t_j))^p$ converges to a signed measure. However, for odd $p$ we typically expect the limit to be zero, see the Appendix for a prototypical example. So to slightly simplify the presentation, we restrict our attention to even $p$.}
\end{remark}


\begin{remark}{\em
	A  notion similar to our definition of $p$-th variation was  introduced by Errami and Russo~\cite{errami2003}, in the (probabilistic and not pathwise) context of stochastic calculus via regularization~\cite{russo2007}. For $p=3$, Errami and Russo prove an It\^o type formula that is similar to the one in  Theorem~\ref{thm:follmer-ito}. However, since they use a definition of the integral $\int_0^t f'(S(s)) \dd S(s)$ that does not take the higher order compensation terms into account, their approach is limited to $p=3$. Gradinaru, Russo, and Vallois~\cite{gradinaru2003}  extended this approach to $p=4$ for functions of a fractional Brownian motion with Hurst index $H\ge 1/4$, a result which relies heavily on the  Gaussian properties of fractional Brownian motion.
	
	The key  ingredient of our approach is to define the integral using compensated Riemann sums which, compared with previous work, drastically simplifies the derivation of the change of variable formula for arbitrary (even) $p$ in a strictly pathwise setting without any use of probabilistic notions of convergence.}
\end{remark}

\subsection{Extension to path-dependent functionals}\label{sec.pathdependent}

  An important generalization of \follmer's pathwise \ito formula is to the case of path-dependent functionals  \cite{CF10B} of paths  $S \in V_2(\pi)$ using Dupire's notion of functional derivative \cite{D09}; see~\cite{cont2012} for an overview. We extend here the functional change of variable formula of Cont and Fourni\'e \cite{CF10B} to functionals of paths  $S \in V_p(\pi)$, where $p$ is any even integer.

Let $D([0,T],\R)$ be the space of c\`{a}dl\`{a}g paths from $[0,T]$ to $\R$ and write
\[
	 \omega_t(s) = \omega(s \wedge t),
\]
for the path $\omega$ stopped at time $t$. Let
\[
	\Lambda_T \assign \{ (t, \omega_t): (t,\omega) \in [0,T] \times D([0,T],\R)\}
\]
be the space of stopped paths. This is a complete metric space equipped with
\[
	d_\infty((t,\omega), (t',\omega')) := \sup_{s \in [0,T]} |\omega(s \wedge t) - \omega'(s \wedge t')| + |t-t'| = \|\omega_t - \omega_{t'}\|_\infty + |t - t'|.
\]
We will also need to stop paths ``right before'' a given time, and set for $t>0$
\[
	\omega_{t-}(s) \assign \begin{cases} \omega(s), & s < t, \\ \mathop{\lim}_{r \uparrow t} \omega(r), & s \ge t, \end{cases}
\]
while $\omega_{0-} \assign \omega_0$.
We first recall some concepts from the non-anticipative functional calculus \cite{CF10B,cont2012}.
\begin{definition}
	A \emph{non-anticipative} functional is a map $F\colon \Lambda_T \to \R$. Let $F$ be a non-anticipative functional.
	\begin{enumerate}
		\item[i.] We write $F \in \bb C_l^{0,0}(\Lambda_T)$ if for all $t \in [0,T]$ the map $F(t,\cdot) \colon D([0,T],\R) \to \R$ is continuous and if for all $(t,\omega) \in \Lambda_T$ and all $\varepsilon > 0$ there exists $\delta > 0$ such that for all $(t',\omega') \in \Lambda_T$ with $t'<t$ and $d_\infty((t,\omega),(t',\omega')) < \delta$ we have $|F(t,\omega) - F(t',\omega')| < \varepsilon$.
		\item[ii.] We write $F \in \bb B(\Lambda_T)$ if for every $t_0 \in [0,T)$ and every $K>0$ there exists $C_{K,t_0}>0$ such that for all $t \in [0,t_0]$ and all $\omega \in D([0,T],\R)$ with $\sup_{s \in [0,t]} |\omega(s)| \le K$ we have $|F(t,\omega)| \le C_{K,t_0}$.
		\item[iii.] $F$ is \emph{horizontally differentiable} at $(t,\omega) \in \Lambda_T$ if its \emph{horizontal derivative}
			\[
				\mc D F(t,\omega) \assign \lim_{h\downarrow 0} \frac{F(t+h,\omega_t) - F(t,\omega_t)}{h}
			\]
			exists. If it exists for all $(t,\omega) \in \Lambda_T$, then $\mc D F$ is a non-anticipative functional.
		\item[iv.] $F$ is \emph{vertically differentiable} at $(t,\omega) \in \Lambda_T$ if its \emph{vertical derivative}
			\[
				\nabla_\omega F(t,\omega) \assign \lim_{h \downarrow 0} \frac{F(t,\omega_t + h \1_{[t,T]}) - F(t,\omega_t)}{h}
			\]
			exists. If it exists for all $(t,\omega) \in \Lambda_T$, then $\nabla_\omega F$ is a non-anticipative functional. In particular, we define recursively $\nabla_\omega^{k+1} F \assign \nabla_\omega \nabla_\omega^k F$ whenever this is well defined.
		\item[v.] For $p \in \N_0$ we say that $F \in \bb C^{1,p}_b(\Lambda_T)$ if $F$ is horizontally differentiable and $p$ times vertically differentiable in every $(t,\omega) \in \Lambda_T$, and if $F, \mc D F, \nabla_\omega^k F \in \bb C^{0,0}_l(\Lambda_T) \cap \bb B(\Lambda_T)$ for $k=1,\dots, p$.
	\end{enumerate}
\end{definition}

Define the piecewise-constant approximation $S^n$ to $S$ along the partition $\pi_n$:
\begin{equation}\label{eq:pw-cst-approx}
	S^n (t) = \sum_{[t_j, t_{j + 1}] \in \pi_n} S (t_{j + 1})
   \1_{[t_j, t_{j + 1})} (t) + S (T) \1_{\{ T \}} (t) .
\end{equation}
Then $\lim_{n \rightarrow \infty} \| S^n - S \|_{\infty} = 0$ whenever $\mathrm{osc}(S,\pi_n) \to 0$. 

\begin{theorem}[Functional change of variable formula for paths with finite $p$-th variation]\label{thm.functional}
  Let $p$ be an even integer, let $F \in \mathbb{C}^{1, p}_b(\Lambda_T)$, and let $S \in V_p(\pi)$ for a sequence of partitions $(\pi_n)$ with vanishing mesh size $|\pi_n|\to 0$. Then the functional change of variable formula
  \[ F (t, S_t) = F (0, S_0) + \int_0^t \mathcal{D}F (s, S_s) \dd s +
     \int_0^t <\nabla F (s, S_s), \dd S (s)> + \frac{1}{p!} \int_0^t
     \nabla_{\omega}^p F (s, S_s) \dd [S]^p (s) \]
  holds, where
  \[ \int_0^t <\nabla F (s, S_s) ,\dd S (s)> \assign \lim_{n \rightarrow
     \infty} \sum_{[t_j, t_{j + 1}] \in \pi_n} \sum_{k=1}^{p-1}
     \frac{1}{k!} \nabla_{\omega}^k F (t_j, S^n_{t_j -}) (S (t_{j + 1} \wedge
     t) - S (t_j \wedge t))^k, \]
     with the piecewise constant approximation $S^n$ as defined in~\eqref{eq:pw-cst-approx}.
\end{theorem}

\begin{proof}
	Since the right hand side is a telescoping sum, we have
	\ba
  		F (t, S^n_t) - F (0, S^n_0) & = \sum_{[t_j, t_{j + 1}] \in \pi_n} (F (t_{j + 1} \wedge t, S^n_{(t_{j + 1} \wedge t) -}) - F (t_j \wedge t, S^n_{(t_j \wedge t) -})) \\
  		&\quad + F(t, S^n_t) - F(t, S^n_{t-}) \\
  		& = \sum_{[t_j, t_{j + 1}] \in \pi_n} (F (t_{j + 1} \wedge t, S^n_{(t_{j + 1} \wedge t) -}) - F (t_j \wedge t, S^n_{(t_j \wedge t) -})) + o (1).
  	\ea
  Consider $j$ with $t_{j + 1} \leqslant t$ and split up the difference as
  follows:
  \begin{equation*}
    F (t_{j + 1}, S^n_{t_{j + 1} -}) - F (t_j, S^n_{t_j -}) = (F (t_{j + 1},
    S^n_{t_{j + 1} -}) - F (t_j, S^n_{t_j})) + (F (t_j, S^n_{t_j}) - F (t_j,
    S^n_{t_j -})) .
  \end{equation*}
  Now $S^n_{t_{j + 1} -} (s) = S^n_{t_j} (s)$ for all $s \in [0, t_{j + 1}]$,
  and therefore the first term on the right hand side is simply
  \[ F (t_{j + 1}, S^n_{t_{j + 1} -}) - F (t_j, S^n_{t_j}) = \int_{t_j}^{t_{j
     + 1}} \mathcal{D}F (r, S^n_{t_j}) \dd r, \]
  from where we easily get (using that the mesh size of $(\pi_n)$ converges to zero)
  \[ \lim_{n \rightarrow \infty} \sum_{[t_j, t_{j + 1}] \in \pi_n} (F (t_{j +
     1} \wedge t, S^n_{(t_{j + 1} \wedge t) -}) - F (t_j \wedge t, S^n_{(t_j
     \wedge t)})) = \int_0^t \mathcal{D}F (r, S_r) \dd r. \]
  It remains to consider the term
	\[
    	F (t_j, S^n_{t_j}) - F (t_j, S^n_{t_j -}) = F (t_j, S^{n, S_{t_j, t_{j + 1}}}_{t_j -}) - F (t_j, S^n_{t_j -}),
  	\]
  where $S_{t_j, t_{j + 1}} \assign S (t_{j + 1}) - S (t_j)$ and $S^{n, x}_{t_j-}
  (s) \assign S^n_{t_j} (s) +\1_{[t_j,T]} (s) x$. By Taylor's
  formula and the definition of the vertical derivative, we have
	\ba
    	F (t_j, S^{n, S_{t_j, t_{j + 1}}}_{t_j -}) - F (t_j, S^n_{t_j -}) & = \sum_{k=1}^p \frac{\nabla_{\omega}^k F (t_j, S_{t_j -}^n)}{k!} (S (t_{j + 1} \wedge t) - S (t_j \wedge t))^k\\
    	&\quad + \frac{1}{(p - 1) !} \int_0^1 \dd \lambda (1 - \lambda)^{p - 1} (S (t_{j + 1} \wedge t) - S(t_j \wedge t))^p \\
    	&\hspace{60pt} \times \left(\nabla_{\omega}^p F (t_j, S_{t_j -}^{n, \lambda S_{t_j, t_{j + 1}}}) -    \nabla_{\omega}^p F (t_j, S_{t_j -}^n) \right).
	\ea
  Now we sum over $[t_j, t_{j + 1}] \in \pi_n$ and see as in Theorem~\ref{thm:follmer-ito} that the correction term vanishes for $n \rightarrow \infty$. Moreover, since $S \in V_p(\pi)$ we have
  	\[
  		\lim_{n \to \infty} \sum_{[t_j, t_{j + 1}] \in \pi_n} \frac{\nabla_{\omega}^p F (t_j, S_{t_j -}^n)}{p!} (S (t_{j + 1} \wedge t) - S (t_j \wedge t))^p = \frac{1}{p!} \int_0^t \nabla_\omega^p F(s,S_s) \dd [S]^p(s),
  	\]
  	see~\cite[Lemma~5.3.7]{cont2012}. Since $F \in \bb C^{0,0}_l(\Lambda_T)$, we have 
  	\[
  		\lim_{n \rightarrow \infty} (F (t, S^n_t) - F (0, S^n_0)) = F (t, S_t) - F (0, S_0),
  	\]
  	which completes the proof.
\end{proof}

\section{Isometry relation and rough-smooth decomposition}\label{sec.isometry}

Given a path (or process)  $S \in V_p(\pi)$ with finite $p$-th variation along the sequence of partitions $(\pi_n)$, the results above may be used to derive a decomposition of regular functionals of $S$ into a rough component with non-zero $p$-th variation along $(\pi_n)$ and a smooth component with zero $p$-th variation along $(\pi_n)$. For $p=2$ such a decomposition was obtained in \cite{ananova2017} and is a pathwise analog of the decomposition of a Dirichlet process into a local martingale and a ``zero energy'' part~\cite{follmer1981b}.

For $\alpha \in (0,1)$ we write $C^\alpha([0,T],\R)$ for the $\alpha$-H\"older continuous paths from $[0,T]$ to $\R$, and $\|\cdot\|_\alpha$ denotes the $\alpha$-H\"older semi-norm.

\subsection{An  `isometry' property of the pathwise integral}
\begin{theorem}[`Isometry' formula]\label{thm.isometry} Let $p \in \N$ be an even integer, let $\alpha > ((1+\tfrac{4}{p})^{1/2}-1)/2$, let $(\pi_n)$ be a sequence of partitions with mesh size going to zero, and let $S \in V_p(\pi) \cap C^\alpha([0,T],\R)$. Let $F\in \mathbb{C}^{1,2}_b(\Lambda_T)$ such that $\nabla_\omega F \in \bb C^{1,1}_b(\Lambda_T)$. Assume furthermore that $F$ is Lipschitz-continuous with respect to $d_\infty$. Then $F(\cdot, S) \in V_p(\pi)$ and
\[
	[F(\cdot,S)]^p(t)=\int_0^t |\nabla_\omega F(s,S_s)|^p \dd[S]^p(s).
\]
\end{theorem}

\begin{proof}
	The proof is similar to the case $p=2$ considered in~\cite{ananova2017}. Indeed, our assumptions allow us to apply~\cite[Lemma~2.2]{ananova2017}, which shows that there exists $C>0$, only depending on $T$, $F$, and $\|S\|_\alpha$, such that for all $0 \le s \le t \le T$
	\begin{equation}\label{eq:F-pvar-pr1}
		R_F(s,t) \assign |F(t,S_t) - F(s, S_s) - \nabla_\omega F(s,S_s) (S(t) - S(t))| \le C |t-s|^{\alpha+\alpha^2}.
	\end{equation}
	Writing also $\gamma_F(s,t) \assign \nabla_\omega F(s,S_s) (S(t) - S(s))$, we obtain
	\begin{align}\label{eq:F-pvar-pr2} \nonumber
		&\sum_{\substack{[t_j, t_{j + 1}] \in \pi_n:\\ t_{j+1} \le t}} |F(t_{j+1}, S_{t_{j + 1}}) - F(t_j, S_{t_j})|^p = \sum_{\substack{[t_j, t_{j + 1}] \in \pi_n:\\ t_{j+1} \le t}} |R_F(t_j, t_{j+1}) + \gamma_F(t_j,t_{j+1})|^p \\ 
		&\hspace{30pt} = \sum_{\substack{[t_j, t_{j + 1}] \in \pi_n:\\ t_{j+1} \le t}} |\gamma_F(t_j,t_{j+1})|^p + \sum_{k=1}^p \binom{p}{k} \sum_{\substack{[t_j, t_{j + 1}] \in \pi_n:\\ t_{j+1} \le t}} R_F(t_j, t_{j+1})^k \gamma_F(t_j,t_{j+1})^{p-k}.
	\end{align}
	Since $S \in V_p(\pi)$ we have
	\begin{equation}\label{eq:F-pvar-pr3}
		\lim_{n \to \infty} \sum_{\substack{[t_j, t_{j + 1}] \in \pi_n:\\ t_{j+1} \le t}} |\gamma_F(t_j, t_{j+1})|^{p} = \int_0^t |\nabla_\omega F(s,S(s))|^p \dd[S]^p(s).
	\end{equation}
	Our result follows once we show that the double sum on the right hand side of~\eqref{eq:F-pvar-pr2} vanishes. For that purpose let $k \in \{1,\dots,p\}$ and write $q_k \assign p/(p-k) \in [1,\infty]$ and let $q_k' = p/k$ be its conjugate exponent. H\"older's inequality yields
	\ba
		& \Big| \sum_{\substack{[t_j, t_{j + 1}] \in \pi_n:\\ t_{j+1} \le t}} R_F(t_j, t_{j+1})^k \gamma_F(t_j,t_{j+1})^{p-k} \Big| \\
		& \le \Big( \sum_{\substack{[t_j, t_{j + 1}] \in \pi_n:\\ t_{j+1} \le t}} |R_F(t_j, t_{j+1})|^{k q_k'} \Big)^{1/q_k'} \Big( \sum_{\substack{[t_j, t_{j + 1}] \in \pi_n:\\ t_{j+1} \le t}} |\gamma_F(t_j, t_{j+1})|^{(p-k) q_k} \Big)^{1/q_k} \\
		& = \Big( \sum_{\substack{[t_j, t_{j + 1}] \in \pi_n:\\ t_{j+1} \le t}} |R_F(t_j, t_{j+1})|^{p} \Big)^{k/p} \Big( \sum_{\substack{[t_j, t_{j + 1}] \in \pi_n:\\ t_{j+1} \le t}} |\gamma_F(t_j, t_{j+1})|^{p} \Big)^{(p-k)/p}.
	\ea
	By~\eqref{eq:F-pvar-pr1} the first sum on the right hand side is bounded by
	\ba
		\Big( \sum_{\substack{[t_j, t_{j + 1}] \in \pi_n:\\ t_{j+1} \le t}} |R_F(t_j, t_{j+1})|^{p} \Big)^{k/p} & \lesssim \Big( \sum_{\substack{[t_j, t_{j + 1}] \in \pi_n:\\ t_{j+1} \le t}} |t_{j+1} - t_j|^{p(\alpha+\alpha^2)} \Big)^{k/p}\\
		& \le (t \times \max\{|t_{j+1} - t_j|^{p(\alpha + \alpha^2) - 1}: [t_j,t_{j+1}] \in \pi_n, t_{j+1} \le t\})^{k/p},
	\ea
	which converges to zero for $n \to \infty$ because $p(\alpha+\alpha^2) > 1$ (which is equivalent to our assumption $\alpha > (\sqrt{1+\frac{4}{p}}-1)/2$) and because $k > 0$. Moreover, by~\eqref{eq:F-pvar-pr3} the sum over $|\gamma_F(t_j, t_{j+1})|^p$ is bounded and this concludes the proof.
\end{proof}

\begin{remark}
	\begin{enumerate}
		\item Keeping the example of the (fractional) Brownian motion in mind, we would typically expect paths in $V_p(\pi)$ to be $(1/p-\kappa)$-H\"older continuous for any $\kappa > 0$. Since for $f(x) = (1+x)^{1/2}$ we have
			\[
				f''(x) = -\frac{1}{4} (1+x)^{-3/2} < 0,
			\]
			we have $f(x) < f(0) + f'(0) x$ for all $x > 0$, and therefore
			\[
				\frac{(1+\frac{4}{p})^{1/2}-1}{2} < \frac{\frac{1}{2}\frac{4}{p}}{2} = \frac{1}{p},
			\]
			which means that in Theorem~\ref{thm.decomposition} we can take $\alpha < 1/p$ and our constraint on the H\"older regularity is not unreasonable.
		\item In fact the complicated constraint on $\alpha$ comes from inequality~\eqref{eq:F-pvar-pr1}, which only gives us a control of order $|t-s|^{\alpha+\alpha^2}$ for $R_F(s,t)$, while $|t-s|^{2\alpha}$ might seem more natural (after all $R_F(s,t)$ is something like the remainder in a first order Taylor expansion). The difficulty is that horizontal differentiability is a very weak notion, which a priori gives us no control on $R_F(s,t)$. To obtain any bounds at all we first need to approximate our path by piecewise linear or piecewise constant paths, and through this approximation procedure we lose a little bit of regularity, see~\cite[Lemma~2.2]{ananova2017} for details. We could improve the control of $R_F(s,t)$ by taking a higher order Taylor expansion (which would require more regularity from $F$), but we do not need this here.
	\end{enumerate}

\end{remark}

\subsection{Pathwise rough-smooth decomposition}

Using the above result we may derive, as in \cite{ananova2017}, a pathwise `signal plus noise' decomposition for regular functionals of paths with strictly increasing $p$-th variation. Let
\[
	\mathbb{C}^{1,p}_b(S)=\{ F(\cdot,S) , F\in \mathbb{C}^{1,p}_b(\Lambda_T) \}\subset V_p(\pi).
\]
The following result extends the pathwise rough-smooth decomposition of paths in $\mathbb{C}^{1,p}_b(S)$, obtained in \cite{ananova2017} for $p=2$, to higher values of $p$. 

\begin{theorem}\label{thm.decomposition}
Let $p \in \N$ be an even integer, let $\alpha>((1+\tfrac{4}{p})^{1/2}-1)/2$, let $(\pi_n)$ be a sequence of partitions with vanishing mesh size $|\pi_n|\to 0$ and let $S\in V_p(\pi) \cap C^\alpha([0,T],\R)$ be a path with strictly increasing $p$-th variation $[S]^p$ along $(\pi_n)$. 
Then any  $X\in \mathbb{C}^{1,p}_b(S)$ admits a unique decomposition $X= X(0)+ A+M$ where  $M(t)=
\int_0^t \phi(s)\dd S(s)$ is a pathwise integral defined as in Theorem~\ref{thm.functional}, and where $[A]^p = 0$.
\end{theorem}

\begin{proof}
Consider two such decompositions $X - X_0 = A + M = \tilde A + \tilde M$. Since $[A]^p = [\tilde A]^p = 0$ and
\[
	|(A - \tilde A)(t) - (A - \tilde A)(s)|^p \lesssim |A(t) - A(s)|^p + |\tilde A(t) - \tilde A(s)|^p,
\]
we get $A - \tilde A \in V_p(\pi)$ and $[A- \tilde A]^p \equiv 0$. But then also $[M - \tilde M]^p = [A - \tilde A]^p \equiv 0$. Now
\[
	M(t) = \int_0^t \nabla_\omega F(s, S_s) \dd S(s),\qquad \tilde M(t) = \int_0^t \nabla_\omega \tilde F(s,S_s) \dd S(s)
\]
for some $F, \tilde F \in C^{1,p}_b(\Lambda_T)$, and by Theorem~\ref{thm.isometry} we have
\[
	0 = [M - \tilde M]^p(T) = \int_0^T | \nabla_\omega(F - \tilde F)(s,S_s)|^p \dd [S]^p(s).
\]
Since $(F - \tilde F)(s,S_s)$ is continuous in $s$ and $[S]^p$ is strictly increasing we have $\nabla_\omega(F - \tilde F)(\cdot, S) \equiv 0$. This means that $M - \tilde M \equiv 0$, and then also $A - \tilde A \equiv 0$.
\end{proof}

\section{Local times and higher order Wuermli formula}

An extension of \follmer's pathwise \ito formula to less regular functions was given by Wuermli~\cite{wuermli1980} in her (unpublished) thesis. Wuermli considered paths with finite quadratic variation which further admit a {\it  local time} along a sequence of partitions, and derive a pathwise change of variable formula  for more general functions that need not be $C^2$. Depending on  the notion of convergence used to define the local time, one then obtains  Tanaka-type change of variable formulas for various classes of functions; convergence in stronger topologies leads to a formula valid for a larger class of functions. Wuermli \cite{wuermli1980} assumed weak convergence in $L^2$ in the space variable (see also \cite{bertoin1987}) and some recent works have extended the approach to other topologies, for example uniform convergence or weak convergence in $L^q$~\cite{perkowski2015,davis2018}. To a certain extent Wuermli's approach can be generalized to our higher order setting, but as we will discuss below in the higher order case we do not expect to have convergence of the pathwise local times in strong topologies.

	To derive the generalization of Wuermli's formula, we consider $f \in C^{p-2}$ with absolutely continuous $f^{(p-2)}$ and
apply the Taylor expansion of order $p-2$ with integral remainder to obtain
\[
	f (b) - f (a) = \sum_{k=1}^{p-2} \frac{f^{(k)}(a)}{k!}(b-a)^k + \int_a^b \frac{f^{(p-1)}(x)}{(p-2)!} (b-x)^{p-2} \dd x.
\]
Assume now that $f^{(p-1)}$ is of bounded variation. Since every bounded variation function $f^{(p-1)}$ is regulated (l\`{a}dl\`{a}g) and therefore has only countably many jumps, its  c\`adl\`ag version is also a weak derivative of $f^{(p-2)}$, and from now on we only work with this version. Since $(b-\cdot)^{p-2}$ is continuous, the integration by parts rule for the Lebesgue-Stieltjes integral applies in the case $b\ge a$ and we obtain
\ba
	\int_a^b \frac{f^{(p-1)}(x)}{(p-2)!} (b-x)^{p-2} \dd x & = f^{(p-1)}(b) \frac{-(b-b)^{p-1}}{(p-1)!} - f^{(p-1)}(a)\frac{-(b-a)^{p-1}}{(p-1)!} \\
	&\quad - \int_{(a,b]} \frac{-(b-x)^{p-1}}{(p-1)!} \dd f^{(p-1)}(x) \\
	& = f^{(p-1)}(a)\frac{(b-a)^{p-1}}{(p-1)!} + \int_{(a,b]} \frac{(b-x)^{p-1}}{(p-1)!} \dd f^{(p-1)}(x).
\ea
Similarly we get for $b < a$
\ba
	\int_a^b \frac{f^{(p-1)}(x)}{(p-2)!} (b-x)^{p-2} \dd x & = - \int_b^a \frac{f^{(p-1)}(x)}{(p-2)!} (b-x)^{p-2} \dd x \\
	& = f^{(p-1)}(a)\frac{(b-a)^{p-1}}{(p-1)!} - \int_{(b,a]} \frac{(b-x)^{p-1}}{(p-1)!} \dd f^{(p-1)}(x),
\ea
and therefore
\ba
	f (b) - f (a) & = \sum_{k=1}^{p-1} \frac{f^{(k)}(a)}{k!}(b-a)^k + \op{sign}(b-a)\int_{\lb a,b\rb} \frac{(b-x)^{p-1}}{(p-1)!}  \dd f^{(p-1)}(x) \\
	& = \sum_{k=1}^{p-1} \frac{f^{(k)}(a)}{k!}(b-a)^k + \op{sign}(b-a)^{p} \int_{\lb a,b\rb } \frac{|b-x|^{p-1}}{(p-1)!}  \dd f^{(p-1)}(x) \\
	& = \sum_{k=1}^{p-1} \frac{f^{(k)}(a)}{k!}(b-a)^k + \int_{\R} \1_{\lb a,b \rb}(x) \frac{\op{sign}(b-a)^p|b-x|^{p-1}}{(p-1)!}  \dd f^{(p-1)}(x),
\ea
with the notation
\[
	\lb a,b \rb = \begin{cases} (a,b], & b \ge a, \\ (b,a], & a \le b. \end{cases}
\]
For any partition $\sigma$ of $[0, T]$, we  define
\[
	L^{\sigma, p - 1}_t(x) \assign \sum_{t_j \in \sigma} \op{sign}(S_{t_{j+1} \wedge t} - S(t_j \wedge t))^p \1_{\lb S(t_j \wedge t), S_{t_{j + 1} \wedge t}\rb} (x) | S(t_{j + 1} \wedge t) - x |^{p - 1}.
\]
To extend Theorem~\ref{thm:follmer-ito} to $S\in V_p(\pi)$, we first note that the following identity holds for any partition $\pi_n$:
\begin{align}\label{eq:exact-lt} \nonumber
	  f (S_t) - f (S_0)  & = \sum_{[t_j, t_{j + 1}] \in \pi_n} \sum_{k=1}^{p-1} \frac{f^{(k)} (S_{t_j})}{k!} (S(t_{j + 1} \wedge t) - S(t_j \wedge t))^k \\
	  &\quad  + \frac{1}{(p-1)!} \int_{\mathbb{R}} L^{\pi_n, p - 1}_t (x) \dd f^{(p - 1)} (x) .
\end{align}
To obtain a change of variable formula for less regular functions, we need the last term to converge as the partition is refined. This motivates the following definition:
\begin{definition}[Local time of order $p$]
	Let $p\in \mathbb{N}$ be an even integer and let $q \in [1,\infty]$. A  continuous path $S \in C([0,T],\R)$ has an  \emph{$L^q$-local time of order $p-1$} along a sequence of
partitions $\pi=(\pi_n)_{n\geq 1}$ if $\op{osc}(S,\pi_n)\to 0$ and
	\[
		L^{\pi_n, p - 1}_t(\cdot) = \sum_{t_j \in \pi}  \1_{\lb S(t_j \wedge t), S_{t_{j + 1} \wedge t}\rb} (\cdot) | S(t_{j + 1} \wedge t) - \cdot |^{p - 1}
	\]
	converges weakly in $L^q(\R)$ to a weakly continuous map $L\colon[0,T]\to L^q(\mathbb{R})$ which we call the \emph{order $p$ local time} of $S$. We denote $ \mc L_p^q(\pi)$ the set of continuous paths $S$ with this property. 
\end{definition}
Intuitively, the limit $L_t(x)$ then measures the rate at which the path $S$ accumulates p-th order variation near $x$. This definition is further justified by the following result, which is a  `pathwise Tanaka formula' \cite{wuermli1980} for paths of arbitrary regularity:
\begin{theorem}[Pathwise `Tanaka' formula for paths with finite p-th order variation]
	Let $p \in 2\N$ be an even integer,  $q \in [1,\infty]$ with conjugate exponent $q'=q/(q-1)$.  Let $f \in C^{p-1}(\R,\R)$ and assume that $f^{(p-1)}$ is weakly differentiable with derivative in $L^{q'}(\R)$. Then for any $S \in \mc L^q_p(\pi)$ the  pointwise limit of compensated Riemann sums
	\[
		\int_0^t f' (S(s)) \dd S(s) \assign \lim_{n \rightarrow \infty} \sum_{[t_j, t_{j + 1}] \in \pi_n} \sum_{k=1}^{p-1} \frac{f^{(k)} (S(t_j))}{k!}
   (S(t_{j + 1} \wedge t) - S(t_j \wedge t))^k
	\]
	exists and the following change of variable formula holds:
	\[
		f(S(t)) - f(S(0)) = \int_0^t f' (S(s)) \dd S(s) + \frac{1}{(p-1)!} \int_\R f^{(p)}(x) L_t(x) \dd x.
	\]
\end{theorem}

\begin{proof}
	The formula~\eqref{eq:exact-lt} is exact and does not involve any error terms. Noting that $L^{q'}(\R) \subset (L^q)^\ast(\R)$ also for $q=\infty$, our assumptions imply that the second term on the right hand side of~\eqref{eq:exact-lt} converges, so the result follows.
\end{proof}

To  justify the name ``local time'' for $L$, we illustrate how $L$ is related to classical definitions of local times by restricting our attention to a particular sequence of partitions \cite{chacon1981,karandikar1983}:
\begin{definition}\label{def:dyadic-lebesgue}
	Let $S \in C([0,T], \R)$. The \emph{dyadic Lebesgue partition} generated by $S$ is defined via $\tau^n_0 \assign 0$ and
	\[
		\tau^n_{j+1} \assign \inf\{ t \ge \tau^n_j: S_t \in 2^{-n} \Z \setminus\{ S_{\tau^n_j}\} \},
	\]
	and then $\pi_n = (\{\tau^n_j : j \in \N_0\} \cap [0,T]) \cup \{T\}$.
\end{definition}

\begin{lemma}\label{lem:lt-upcrossings}
	Let $p \in \N$ be even, let $S \in C([0,T], \R)$ and let $(\pi_n)$ be the dyadic Lebesgue partition generated by $S$. Given an interval $[a,b]$ we write $U_t([a,b])$ for the number of upcrossings of $[a,b]$ that $S$ performs until time $t$. Let $x \in \R$ and let $I^n_k = (k 2^{-n},(k+1)2^{-n}]$ be the unique dyadic interval of generation $n$ with $x \in I^n_k$. Then
	\[
		L^{\pi_n}_t(x) = (|(k+1)2^{-n} - x|^{p-1} + |x - k2^{-n}|^{p-1})U_t(I^n_k) + O(2^{-n(p-1)}).
	\]
\end{lemma}

\begin{proof}
	We have $\1_{\lb S_{\tau_j^n}, S_{\tau^n_{j + 1}}\rb} (x) \neq 0$ if either $S_{\tau^n_j} = k 2^{-n}$ and $S_{\tau^n_{j+1}} = (k+1) 2^{-n}$ (i.e. $S$ performs an upcrossing of $I^n_k$), or $S_{\tau^n_j} = (k+1) 2^{-n}$ and $S_{\tau^n_{j+1}} = k 2^{-n}$ (i.e. $S$ performs a downcrossing of $I^n_k$). In the first case we have to add $|(k+1)2^{-n} - x|^{p-1}$ to $L^{\pi_n}_t(x)$, and in the second case we add $(-1)^p | x - k 2^{-n}|^{p-1} = | x - k 2^{-n}|^{p-1}$. Therefore, we obtain
\[
	L^{\pi_n}_t(x) = |(k+1)2^{-n} - x|^{p-1} U_t(I^n_k) + |x - k2^{-n}|^{p-1} D_t(I^n_k) + O(2^{-n(p-1)}),
\]
and since up- and downcrossings of $I^n_k$ differ by at most one, our claim follows.
\end{proof}

Note that the expression for $L^{\pi_n}_t$ strongly fluctuates on $I^n_k$. For $x \simeq k2^{-n}$ and $x \simeq (k+1) 2^{-n}$ the factor in front of $U_t(I^n_k)$ is $\simeq 2^{-n(p-1)}$, while for $x = (2k+1) 2^{-n-1}$ we get the factor $2^{-n(p-1)} 2^{p-2}$. Therefore, we do not expect $L^{\pi_n}_t(x)$ to converge uniformly or even pointwise in $x$ as $n\to \infty$ (unless if $p=2$).

\begin{lemma}
	In the setting of Lemma~\ref{lem:lt-upcrossings} set
	\[
		\tilde L^{\pi_n}_t(x) \assign \sum_{k \in \Z} 2^{-n(p-1)} U_t(I^n_k) \1_{I^n_k}(x).
	\]
	Let $q \in (1,\infty)$. If $\tilde L^{\pi_n}_t$ converges weakly in $L^q(\R)$ to a limit $\tilde L_t$, then $L^{\pi_n}_t$ converges weakly in $L^q(\R)$ to $(2/p) \tilde L_t$.
\end{lemma}

\begin{proof}
	Let us introduce an averaging operator,
	\[
		(\mc A_n f)(x) \assign \sum_{k \in \Z} 2^n \int_{I^n_k} f(y) \dd y \, \1_{I^n_k}(x).
	\]
	Since
	\[
		\int_{I^n_k} (|(k+1)2^{-n} - x|^{p-1} + |x - k2^{-n}|^{p-1}) \dd x = 2 \int_0^{2^{-n}} x^{p-1} \dd x = \frac{2}{p} 2^{-n p},
	\]
	we have $\tilde L^{\pi_n}_t = \tfrac{p}{2} \mc A_n L^{\pi_n}_t + O(2^{-n(p-1)})$, with a compactly supported remainder $O(2^{-n(p-1)})$. We claim that if $(f_n)$ is a sequence of functions for which $\mc A_n f_n$ converges weakly in $L^q(\R)$ and for which $|f_n| \le C |\mc A_n f_n|$, then also $(f_n)$ converges weakly in $L^q(\R)$ to the same limit, which will imply our claim. To show this, let $f$ be the limit of $\mc A_n$ and let $g \in L^{q'}(\R)$. We have $\langle \mc A_n \varphi, \psi \rangle = \langle \mc A_n \varphi, \mc A_n \psi \rangle = \langle \varphi, \mc A_n \psi \rangle$ for all $\varphi, \psi$, and therefore
	\ba
		|\langle f_n - f, g \rangle| & \le |\langle f_n - \mc A_n f_n, g \rangle| + |\langle \mc A_n f_n - f, g \rangle| \\
		& = |\langle f_n, g - \mc A_n g \rangle| + |\langle \mc A_n f_n - f, g \rangle| \\
		& \le \| f_n \|_{L^q} \| g - \mc A_n g \|_{L^{q'}} + |\langle \mc A_n f_n - f, g \rangle|.
	\ea
	The second term on the right hand side converges to zero by assumption. For the first term we note that by assumption $\| f_n \|_{L^q} \le \| \mc A_n f_n \|_{L^q}$, which is uniformly bounded in $n$ because $(\mc A_n f_n)$ converges weakly in $L^q$. The proof is therefore complete once we show that $\lim_{n \to \infty} \| g - \mc A_n g \|_{L^{q'}} = 0$ for all $g \in L^{q'}$. But this easily follows from the fact that the continuous and compactly supported functions are dense in $L^{q'}$.
\end{proof}
In fact, we conjecture that, for fractional Brownian motion, this notion of local time defined along the dyadic Lebesge partition coincides, up to a constant, with the usual concept of local time defined as the density of the occupation measure:
\begin{conjecture}
	Let $B$ be the fractional Brownian motion with Hurst parameter $H \in (0,1)$, and let $(\pi_n)$ be the dyadic Lebesgue partition generated by $B$. Let $I^n_k$ and $U_t$ be as in Lemma~\ref{lem:lt-upcrossings} (where now we count the upcrossings of $B$ instead of $S$). We conjecture that
	\[
		\tilde{L}^{\pi_n}_t(x) \assign \sum_{k \in \Z} 2^{-n(1/H-1)} U_t(I^n_k) \1_{I^n_k}(x)
	\]
almost-surely	converges uniformly in $(t,x) \in [0,T] \times \R$ to $\ell_t(x) \E[|B_1|^{1/H}]/2$, where $\ell$ is the local time of $B$, i.e. the Radon-Nikodym derivative of the occupation measure $A \mapsto \int_0^t \1_A(B(s)) \dd s$ with respect to the Lebesgue measure, see e.g.~\cite{Biagini2008}. In particular, for any even integer $p\in 2\mathbb{N}$,  $B \in \mc L^{p-1}_q(\pi_n)$ for any $q \in (1,\infty)$.
	\end{conjecture}
	This result is well known for $H = 1/2$, see e.g.~\cite{chacon1981,perkowski2015}.
	In the general case $H \in (0,1)$, it is natural to expect that
	\[
		\mu^n([0,t]) \assign \sum_{j=0}^\infty 2^{-n/H} \1_{\tau^n_{j+1} \le t} \xrightarrow{n \to \infty} [B]^{1/H}_t = \E[|B_1|^{1/H}] t,
	\]
	which would be an extension of the convergence result of~\cite{Rogers1997} from deterministic partitions to the Lebesgue partition generated by $B$. Moreover, we know that the local time $\ell$ of the fractional Brownian motion satisfies
	\[
		\ell_t(x) = \lim_{n \to \infty} \sum_{k \in \Z} 2^n \int_0^t \1_{I^n_k}(B_s) \dd s \1_{I^n_k}(x).
	\]
	If we \emph{formally} replace the Lebesgue measure in the integral by $\E[|B_1|^{1/H}]^{-1} \mu^n$, then we get
	\ba
		\ell_t(x) & = \E[|B_1|^{1/H}]^{-1}\lim_{n \to \infty}  \sum_{k \in \Z} 2^n \int_0^t \1_{I^n_k}(B_s) \mu^n(\dd s) \1_{I^n_k}(x) \\
		& = \E[|B_1|^{1/H}]^{-1}\lim_{n \to \infty} \sum_{k \in \Z} 2^{n - n/H} \sum_{j: \tau^n_{j+1} \le t} \1_{I^n_k}(B_{\tau^n_j}) \1_{I^n_k}(x) \\
		& = \E[|B_1|^{1/H}]^{-1}\lim_{n \to \infty} \sum_{k \in \Z} 2^{n - n/H} (D_t(I^n_k) + U_t(I^n_{k+1})) \1_{I^n_k}(x),
	\ea
	and if we further assume that $2^{n - n/H} |U_t(I^n_{k+1}) - U_t(I^n_k)|\to 0$ then our conjecture formally follows. 

If the conjecture holds, then  for any $p\in 2\N$ and $B$ a typical sample path of the fractional Brownian motion with Hurst index $1/p$ and $f \in C^{p-1}$ with weak $p$-th derivative $f^{(p)} \in L^q$ for any $q \in (1,\infty)$:
	\begin{equation}\label{eq:occupation-ito}
		f(B(t)) - f(B(0)) = \int_0^t f' (B(s)) \dd B(s) + \frac{\E[|B_1|^p]}{p!} \int_\R f^{(p)}(x) \ell_t(x) \dd x,
	\end{equation}
	where $\ell$ is the local time of $B$ and
	\[
		\int_0^t f' (B(s)) \dd B(s) \assign \lim_{n \rightarrow \infty} \sum_{[t_j, t_{j + 1}] \in \pi_n} \sum_{k=1}^{p-1} \frac{f^{(k)} (B(t_j))}{k!}
   (B(t_{j + 1} \wedge t) - B(t_j \wedge t))^k.
	\]
	By Theorem~\ref{thm:follmer-ito} the formula holds for $f \in C^p$, because then
	\[
		\frac{\E[|B_1|^p]}{p!} \int_\R f^{(p)}(x) \ell_t(x) \dd x = \frac{\E[|B_1|^p]}{p!} \int_0^t f^{(p)}(S(s)) \dd s = \frac{1}{p!} \int_0^t f^{(p)}(S(s)) \dd [S]^p_s,
	\]
	which adds further credibility to our conjecture.

\section{Extension to multidimensional paths}\label{sec.multidimensional}
As in the case $p=2$, the set $V_p(\pi)$ is not stable under linear combinations: for $S_1,S_2\in V_p(\pi)$, expanding $( (S_1(t_{j + 1}) - S_1(t_j)+ S_2(t_{j + 1}) - S_2(t_j))^p$ yields many cross terms whose sum cannot be controlled in general  as the  partition is refined.
The extension of Definition \ref{def:p-var} to vector-valued functions $S=(S_1,...,S_d)$ therefore requires some care. The original approach of \follmer~\cite{follmer1981} was to require that $S_i, S_i+S_j \in V_p(\pi)$. We propose here a slightly different formulation, which is equivalent to \follmer's construction for $p=2$ but easier to relate to other approaches, such as  rough path integration.

\subsection{Tensor formulation}

Define $T_p(\mathbb{R}^d) =\mathbb{R}^d\otimes ...\otimes \mathbb{R}^d$ as the space of $p$-tensors on  $\mathbb{R}^d$.
A \emph{symmetric $p$-tensor}  is a tensor $T\in T_p(\mathbb{R}^d)$ that is invariant under any permutation   $\sigma$ of its arguments:
$$
\forall  (v_{1},v_{2},\ldots ,v_{p})\in (\mathbb{R}^d)^p,\qquad
 {\displaystyle T(v_{1},v_{2},\ldots ,v_{p})=T(v_{\sigma 1},v_{\sigma 2},\ldots ,v_{\sigma p})}.
$$ 
 The coordinates $(T_{i_{1}i_{2}\cdots i_{p}})$ of a symmetric tensor of order $p$  satisfy
$$
{\displaystyle T_{i_{1}i_{2}\cdots i_{p}}=T_{i_{\sigma 1}i_{\sigma 2}\cdots i_{\sigma p}}.}$$
The space ${\rm Sym}_p(\mathbb{R}^d)$ of symmetric tensors of order $p$ on  $\mathbb{R}^d$ is naturally isomorphic to the dual of the space $\mathbb{H}_p[X_1,...,X_d]$ of homogeneous polynomials of degree $p$ on $\mathbb{R}^d$.  We set  ${\rm Sym}_0(\mathbb{R}^d):=\mathbb{R}.$

An important example of a symmetric $p$-tensor  on $\mathbb{R}^d$ is given by the $p$-th order derivative of  a smooth  function:
$$\forall f\in C^p(\mathbb{R}^d,\mathbb{R}), \forall x\in \mathbb{R}^d: \qquad \nabla^p f(x)\in {\rm Sym}_p(\mathbb{R}^d).$$ 
The symmetry property is obtained by repeated application of Schwarz's lemma.

We define $\mathbb{S}_p(\mathbb{R}^d)$ as the direct sum of ${\rm Sym}_k(\mathbb{R}^d)$ for $k = 0,1,2,...,p$:
$$ \mathbb{S}_p(\mathbb{R}^d) =\bigoplus _{k=0}^{p} {\rm Sym}_{k}(\mathbb{R}^d).
$$
The space $\mathbb{S}_p(\mathbb{R}^d)$  is naturally isomorphic to the dual of the space $\mathbb{R}_p[X_1,...,X_d]$ 
of polynomials of degree $\leq p$ in $d$ variables, which  defines a bilinear product
$$ \langle \cdot,\cdot \rangle \colon \mathbb{S}_p(\mathbb{R}^d)\times \mathbb{R}_p[X_1,...,X_d] \to \mathbb{R}.$$
Slightly abusing notation, we also write $\langle \cdot, \cdot \rangle$ for the canonical inner product on $T_p(\R^d)$. Consider now a continuous $\mathbb{R}^d$-valued path $S\in C([0,T],\mathbb{R}^d)$ and a sequence of partitions $\pi_n=\{t_0^n, \dots, t^n_{N(\pi_n)}\}$ with $t_0^n=0<...< t^n_k<...< t^n_{N(\pi_n)}=T$. Then 
\[
		\mu^n \assign \sum_{[t_j,t_{j+1}] \in \pi_n}  \delta (\cdot - t_j)  \underbrace{(S(t_{j + 1}) - S(t_j))\otimes  \dots \otimes(S(t_{j + 1}) - S(t_j))}_{p\ \rm times} 
	\]
	defines a tensor-valued measure on $[0,T]$ with values in ${\rm Sym}_p(\mathbb{R}^d)$. 
This space of measures is in duality with the space $C([0,T],\mathbb{H}_p[X_1,...,X_d] )$ of continuous functions taking values in homogeneous polynomials of degree $p$, i.e. homogeneous polynomials of degree $p$ with continuous time-dependent coefficients.

\begin{definition}[$p$-th variation of a multidimensional function]\label{def:vector} Let $p\in \mathbb{N}$ be even, let $S \in C([0,T],\R^d)$ be a  continuous path and let $\pi=(\pi_n)_{n\geq 1}$ be a sequence of
partitions of $[0,T]$.
Consider the sequence of tensor-valued measures
	\[
		\mu^n \assign \sum_{[t_j,t_{j+1}] \in \pi_n}  \delta (\cdot - t_j) (S(t_{j + 1}) - S(t_j))^{\otimes p}.
	\]
	We say that $S$ has a $p$-th variation along  $\pi=(\pi_n)_{n\geq 1}$ if $\op{osc}(S,\pi_n)\to 0$ and there exists a ${\rm Sym}_p(\mathbb{R}^d)$--valued measure $\mu_S$ without atoms such that for all $f\in C([0,T], \mathbb{H}_p[X_1,...,X_d])$
	\ba
		\lim_{n\to\infty} \int_{0}^T\langle f,\dd \mu_n \rangle = \lim_{n\to\infty} \sum_{[t_j,t_{j+1}] \in \pi_n} \langle f( t_j), (S(t_{j + 1}) - S(t_j))^{\otimes p} \rangle = \int_{0}^T\langle f,\dd \mu_S\rangle.
	\ea
	In that case we write $S \in V_p(\pi)$ and  we call $[S]^p\colon [0,T]\to  {\rm Sym}_p(\mathbb{R}^d) $ defined by $$[S]^p(t) := \mu([0,t])$$  the \emph{$p$-th variation} of $S$.
\end{definition}

By analogy with the positivity property of symmetric matrices, we say that a symmetric $p$-tensor $T \in {\rm Sym}_p(\mathbb{R}^d)$ is {\rm positive} if
\ba
	\langle T,v\otimes ...\otimes v \rangle \geq 0, \qquad \forall v\in\mathbb{R}^d.
\ea
We denote the set of positive symmetric $p$-tensors by ${\rm Sym}_p^+(\mathbb{R}^d)$.
For $T,\tilde T\in {\rm Sym}_p(\mathbb{R}^d)$ we write $T\geq \tilde T$ if $T-\tilde T\in {\rm Sym}^+_p(\mathbb{R}^d)$. This defines a partial order on ${\rm Sym}_p(\mathbb{R}^d)$.

\begin{property}\label{property:fv} Let $S \in V_p(\pi)\cap C([0,T],\mathbb{R}^d)$. Then 
\begin{enumerate}
\item[(i)] $[S]^p$ has finite variation and is increasing in the sense of the partial order on ${\rm Sym}_p(\mathbb{R}^d)$:
	\[
		[S]^p(t+h) - [S]^p(t) \in {\rm Sym}^+_p(\mathbb{R}^d), \qquad \forall \, 0\leq t\leq t+h\leq T.
	\]
$$(ii)\qquad\forall t\in [0,T],\qquad  \sum_{ \pi_n} (S(t_{j + 1}\wedge t) - S(t_j\wedge t))^{\otimes p} \quad\mathop{\to}^{n\to\infty} \quad[S]^p(t).\qquad$$
\end{enumerate}
\end{property}

\begin{proof}
	Let $v \in \R^d$. Before passing to the limit, the function
	\[
		\sum_{\substack{[t_j,t_{j+1}] \in \pi_n: \\ t_{j} \le t}} \langle v^{\otimes p}, (S(t_{j + 1}) - S(t_j))^{\otimes p} \rangle = \sum_{\substack{[t_j,t_{j+1}] \in \pi_n: \\ t_{j} \le t}} | v \cdot (S(t_{j + 1}) - S(t_j))|^{p} 
	\]
	is increasing in $t$, and therefore it defines a finite (positive) measure. By assumption, this measure converges weakly to the measure defined by $(a,b] \mapsto \int_0^T\langle \1_{(a,b]} v^{\otimes p}, \dd \mu_S\rangle$. In particular, we have
	\[
		\langle v^{\otimes p}, [S]^p(t+h) - [S]^p(t)\rangle = \int_{0}^T \langle \1_{(t,t+h]} v^{\otimes p}, \dd \mu_S\rangle \ge 0.
	\]
	Thus, $\langle v^{\otimes p}, [S]^p\rangle$ is increasing for all $v \in \R^d$, and from here it is easy to see that $[S]^p$ has finite variation (apply e.g. polarization to go from $v^{\otimes p}$ to $v_1 \otimes \dots \otimes v_p$).
\end{proof}

\begin{theorem}[Change of variable formula for paths with finite $p$-th variation]\label{thm:follmer-ito-multidim}
	Let $p \in \N$ be  even, let $(\pi_n)$ be a  sequence of partitions of $[0,T]$ and let $S \in V_p(\pi)\cap C([0,T],\mathbb{R}^d)$. Then for all $f \in C^p(\R^d,\R)$ the  limit of compensated Riemann sums
	\[
		\int_0^t \langle \nabla f(S(s)) , \dd S(s) \rangle \assign \lim_{n \rightarrow \infty} \sum_{[t_j, t_{j + 1}] \in \pi_n} \sum_{k=1}^{p-1} \frac{1}{k!} \langle \nabla^k f(S(t_j)), (S(t_{j + 1} \wedge t) - S(t_j \wedge t))^{\otimes k} \rangle
	\]
	exists for every $t\in [0,T]$ and satisfies the 
	pathwise change of variable formula:
	\[
		f(S(t)) - f(S(0)) = \int_0^t \langle \nabla f(S(s)) , \dd S(s) \rangle + \frac{1}{p!} \int_0^t \langle \nabla^p f(S(s)), \dd [S]^p(s) \rangle.
	\]
\end{theorem}

\begin{proof}
	The proof follows similar ideas to the case $p=2$. By applying a Taylor expansion at order $p$ to the increments of $f(S)$ along the partition,  we obtain
	\begin{align}\label{eq:ito-pr1-multidim}
  		f (S(t)) - f (S(0)) & = \sum_{[t_j, t_{j + 1}] \in \pi_n} (f (S(t_{j + 1}\wedge t)) - f (S(t_j \wedge t)))\\ \nonumber
  		& = \sum_{[t_j, t_{j + 1}] \in \pi_n} \sum_{k=1}^p \frac{1}{k!} \langle \nabla^{k}f (S(t_j)), (S(t_{j + 1}\wedge t) - S(t_j \wedge t))^{\otimes k} \rangle \\ \nonumber
  		&\quad + \sum_{[t_j, t_{j + 1}] \in \pi_n} \int_0^1 \dd \lambda \frac{(1 - \lambda)^{p - 1}}{(p - 1) !}  \\ \nonumber
  &\hspace{70pt}\times \big\langle \big(\nabla^{p} f(S(t_j) + \lambda (S(t_{j + 1} \wedge t) - S(t_j \wedge t))) - \nabla^{p} f (S(t_j))\big), \\ \nonumber
  &\hspace{120pt} (S(t_{j + 1} \wedge t) - S(t_j \wedge t))^{\otimes p} \big\rangle .
	\end{align}
	As in the proof of Theorem~\ref{thm:follmer-ito} we assume that $f$ is compactly supported and use this to show that the remainder on the right hand side vanishes as $n \rightarrow \infty$. 
	Since $S \in V_p(\pi)$ we know that
	\[
		\lim_{n\to \infty} \sum_{[t_j, t_{j + 1}] \in \pi_n} \frac{1}{p!} \langle \nabla^k f (S(t_j)), (S(t_{j + 1}\wedge t) - S(t_j \wedge t))^{\otimes p} \rangle = \frac{1}{p!} \int_0^t \langle \nabla^p f(S(s)), \dd [S]^p(s)\rangle,
	\]
	and therefore we obtain from~\eqref{eq:ito-pr1-multidim}
	\ba
		&\lim_{n \to \infty} \sum_{[t_j, t_{j + 1}] \in \pi_n} \sum_{k=1}^{p-1} \frac{1}{k!} \langle \nabla^{k}f (S(t_j)), (S(t_{j + 1}\wedge t) - S(t_j \wedge t))^{\otimes k} \rangle \\
		&\hspace{50pt} = f(S(t)) - f(S(0)) - \frac{1}{p!} \int_0^t f^{(p)}(S(s)) \dd [S]^p(s),
	\ea
	and we simply define $\int_0^t \langle \nabla f(S(s)), \dd S(s)\rangle$ as the limit on the left hand side.
\end{proof}

\subsection{Relation with rough path integration}\label{sec.multidimensional-rp}

To explain the link between F\"ollmer's pathwise It\^o integral and rough path integration~\cite{lyons1998}, Friz and Hairer \cite[Chapter~5.3]{FrizHairer}  introduced the notion of (second order) reduced rough paths: 
\begin{definition}\label{def:reduced-rp-2nd-order}
	Let $\alpha \in (1/3,1/2)$. We set $\Delta_T \assign \{(s,t): 0 \le s \le t \le T\}$. A \emph{reduced rough path} of regularity $\alpha$ is a pair $(X, \bb X)\colon \Delta_T \to \R^d \oplus \bb {\rm Sym}_2(\R^d)$, such that
	\begin{itemize}
		\item[(i)] there exists $C>0$ with
			\[
				|X_{s,t}| + \sqrt{|\bb X_{s,t}|} \le C |t-s|^\alpha, \qquad (s,t) \in \Delta_T;
			\]
		\item[(ii)] the \emph{reduced Chen relation} holds
			\[
				\bb X_{s,t} - \bb X_{s,u} - \bb X_{u,t} = {\rm Sym}(X_{s,u} \otimes X_{u,t}), \qquad (s,u), (u,t) \in \Delta_T,
			\]
			where ${\rm Sym}(\cdot)$ denotes the symmetric part.
	\end{itemize}
\end{definition}
Friz and Hairer \cite{FrizHairer}  also show that, for  any $S \in V_2(\pi)$ there is a canonical candidate for a reduced rough path. Indeed, the pair
\[
	X_{s,t} \assign S(t) - S(s),\qquad \bb X_{s,t} \assign \frac{1}{2} X_{s,t} \otimes X_{s,t} - \frac12 ([S]^2(t) - [S]^2(s))
\]
satisfies the reduced Chen relation. But in general we do not know anything about the H\"older regularity of $S \in V_2(\pi)$, because for any continuous path $S$ there exists a sequence of partitions $(\pi_n)$ with $S \in V_2(\pi)$ and $[S]^2 \equiv 0$, see~\cite{freedman2012}. If however we take the dyadic Lebesgue partition $(\pi_n)$ generated by $S$ as in Definition~\ref{def:dyadic-lebesgue} and if $S \in V_2(\pi)$, then it follows from~\cite[Lemme~1]{bruneau1979}\footnote{Note that for $\lambda>0$ the path $S$ has finite $q$-variation if and only if $\lambda^{-1} S$ has finite $q$-variation, and therefore we can assume that $\lambda = 1$ in~\cite[Lemme~1]{bruneau1979}.} that $S$ has finite $q$-variation for any $q>2$. So in that case every $S \in V_2(\pi)$ corresponds to a reduced rough path with $p$-variation regularity. 
Rather than adapting Definition~\ref{def:reduced-rp-2nd-order} from H\"older to $p$-variation regularity, we directly introduce a concept of higher-order reduced rough paths. For that purpose we first define the concept of {\it control function}:
\begin{definition}
	A \emph{control function} is a continuous map $c \colon \Delta_T \to \R_+$ such that $c(t,t)=0$ for all $t \in [0,T]$ and such that $c(s,u) + c(u,t) \le c(s,t)$ for all $0 \le s \le u \le t \le T$.
\end{definition}

A function $f \colon [0,T] \to \R^d$ has finite $p$-variation if and only if there exists a control function $c$ with $|f(t) - f(s)|^p \le c(s,t)$, and in that case $\| f \|_{\op{p-var}} \le c(0,T)^{1/p}$.

\begin{definition}\label{def:reduced-rp}
	Let $p \ge 1$. A \emph{reduced rough path} of finite $p$-variation is a tuple
	\[
		\bb X = (1,\bb X^1, \dots, \bb X^{\lfloor p \rfloor})\colon \Delta_T \longrightarrow \bb S_{\lfloor p\rfloor}(\R^d),
	\]
	such that
	\begin{itemize}
		\item[(i)]\label{def:reduced-rp-i} there exists a control function $c$ with
			\[
				\sum_{k=1}^{\lfloor p\rfloor} |\bb X^k_{s,t}|^{p/k} \le c(s,t), \qquad (s,t) \in \Delta_T;
			\]
		\item[(ii)] the \emph{reduced Chen relation} holds
			\[
				\bb X_{s,t} = {\rm Sym}(\bb X_{s,u} \otimes \bb X_{u,t}), \qquad (s,u), (u,t) \in \Delta_T,
			\]
			where the symmetric part of $T \in T_k(\R^d)$ is defined as
			\[
				{\rm Sym}(T) \assign \frac{1}{k!} \sum_{\sigma \in \mathfrak{S}_k} \sigma T,\qquad \sigma T(v_1, \dots, v_k) \assign T(v_{\sigma1},\dots, v_{\sigma k}),
			\]
			with the group of permutations $\mathfrak S_k$ of $\{1, \dots, k\}$.
	\end{itemize}
\end{definition}

\begin{lemma}\label{lem:reduced-rp-from-Vp}
	Let $S \in C([0,T],\R^d)$ and let $(\pi_n)$ be the dyadic Lebesgue partition generated by $S$. Let $p \ge 1$ and assume that $S \in V_p(\pi)$. Then for any $q>p$ with $\lfloor q \rfloor = \lfloor p \rfloor$ we obtain a reduced rough path of finite $q$-variation by setting $\bb X^0_{s,t} \assign 1$,
	\begin{gather*}
		\bb X^k_{s,t} \assign \frac{1}{k!} (S(t) - S(s))^{\otimes k},\qquad k=1,\dots, \lfloor p \rfloor - 1,\\
		\bb X^{\lfloor p \rfloor}_{s,t} \assign \frac{1}{\lfloor p \rfloor!} (S(t) - S(s))^{\otimes \lfloor p \rfloor} - \frac{1}{\lfloor p \rfloor!} ([S]^p(t) - [S]^p(s)).
	\end{gather*}
\end{lemma}

\begin{proof}
	Let $q>p$. As discussed above we know that $S$ has finite $q$-variation, so let us start by setting
	\[
		\tilde{c}(s,t) \assign \|S \|_{\op{q-var},[s,t]}^q \assign \sup_{\pi\in \Pi([s,t])} \sum_{[t_j, t_{j + 1}] \in \pi} |S(t_{j + 1}) - S(t_j)|^q,\qquad (s,t) \in \Delta_T,
	\]
	which is a control function such that
	\[
		\sum_{k=1}^{\lfloor p\rfloor} |\bb X^k_{s,t}|^{q/k} \le C_{d,p} \big( \tilde{c}(s,t) + |[S]^p(t) - [S]^p(s)|^{q/\lfloor p \rfloor} \big),
	\]
	with a constant $C_{d,p} >0$ that only depends on the dimension $d$ and on $p$. By Property~\ref{property:fv} the path $[S]^p$ has finite variation and therefore it also has finite $q/\lfloor p \rfloor$-variation, so
	\[
		\tilde{\tilde{c}}(s,t) \assign \| [S]^p \|_{\op{q/\lfloor p \rfloor-var},[s,t]}^{q/\lfloor p \rfloor}
	\]
	defines another control function. Therefore, $c(s,t) \assign C_{d,p} ( \tilde{c}(s,t) + \tilde{\tilde{c}}(s,t))$ is a control function for which the analytic property (i) in Definition~\ref{def:reduced-rp-i} holds.
	
	To show the reduced Chen relation let us write $\mathfrak{S}_{\ell, k}$ for $0 \le \ell, k$ for the \emph{shuffles} of words of length $\ell,k$, i.e. for those permutations $\sigma \in \mathfrak{S}_{\ell + k}$ which satisfy $\sigma i < \sigma j$ for all $1 \le i < j \le \ell$ respectively $\ell+1 \le i < j \le k$. Note that there are $\binom{\ell+k}{\ell}$ shuffles in $\mathfrak{S}_{\ell,k}$. We have for $k < \lfloor p \rfloor$
	\ba
		\bb X_{s,t}^k & = \frac{1}{k!} (S(t) - S(s))^{\otimes k} = \frac{1}{k!} (S(t) - S(u) + S(u) - S(s))^{\otimes k} \\
		& = \frac{1}{k!} \sum_{\ell = 0}^k \sum_{\sigma \in \mathfrak{S}_{\ell, k-\ell}} \sigma\big((S(u) - S(s))^{\otimes \ell} \otimes (S(t) - S(u))^{\otimes (k-\ell)}\big),
	\ea
	where we set $v^{\otimes 0} \assign 1$ for all $v \in R^d$. On the other hand, if $\mathcal{P}_k$ denotes the projection onto $T_k(\R^d)$, then for $k < \lfloor p \rfloor$
	\ba
		\mc P_k ({\rm Sym}(\bb X_{s,u} \otimes \bb X_{u,t})) & = \sum_{\ell=0}^k {\rm Sym}( \bb X^\ell_{s,u} \otimes \bb X^{k-\ell}_{u,t}) \\
		& = \sum_{\ell=0}^k \frac{1}{\ell! (k-\ell)!} {\rm Sym}\big((S(u) - S(s))^{\otimes \ell} \otimes (S(t) - S(u))^{\otimes(k-\ell)}\big) \\
		& = \sum_{\ell=0}^k \frac{1}{\ell! (k-\ell)!} \binom{k}{\ell}^{-1} \sum_{\sigma \in \mathfrak{S}_{\ell, k - \ell}} \sigma((S(u) - S(s))^{\otimes \ell} \otimes (S(t) - S(u))^{\otimes(k-\ell)}\big) \\
		& = \bb X^k_{s,t},
	\ea
	which proves the reduced Chen relation for $k < \lfloor p \rfloor$. For $k = \lfloor p \rfloor$ we get the same relation by noting that $[S]^p$ is already symmetric and therefore ${\rm Sym}([S]^p(t) - [S]^p(s)) = [S]^p(t) - [S]^p(s)$.
\end{proof}

The following space of (higher order) controlled paths in the sense of Gubinelli~\cite{gubinelli2004} is defined for example in~\cite[Chapter~4.5]{FrizHairer}. We adapt the definition to paths that are controlled in the $p$-variation sense by a reduced rough path. If $\ell < k$ and $T \in T_\ell$, $\tilde T \in T_k$, then we interpret
	\[
		\langle T, \tilde T \rangle \in T_{k - \ell}, \qquad \langle T, \tilde T \rangle (v_1, \dots, v_{k-\ell}) \assign \langle T \otimes (v_1 \otimes \dots \otimes v_{k - \ell}), \tilde T \rangle,
	\]
	and similarly for $\langle \tilde T, T\rangle$.

\begin{definition}\label{def:controlled}
	Let $p \ge 1$ and let $\bb X$ be a reduced rough path of finite $p$-variation. A path
	\[
		Y = (Y^0, Y^1, \dots,  Y^{\lfloor p \rfloor}) \in C([0,T],\bb S_{\lfloor p\rfloor}(\R^d))
	\]
	is  \emph{controlled} by $\bb X$ if there exists a control function $c$ such that
	\[
		 \sum_{\ell = 1}^{\lfloor p \rfloor} \Big| Y^\ell(t) - \sum_{k=\ell}^{\lfloor p \rfloor} \langle Y^k(s), \bb X^{k-\ell}_{s,t} \rangle \Big|^{\frac{p}{\lfloor p \rfloor - \ell + 1}} \le c(s,t),\qquad (s,t) \in \Delta_T.
	\]
	In that case we write $Y \in \mc D^{\lfloor p \rfloor /p}_{\bb X}([0,T])$.
\end{definition}

\begin{example}\label{ex:controlled-f}
	Let $p \ge 1$, let $S$, $\bb X$ and $q$ be as in Lemma~\ref{lem:reduced-rp-from-Vp}, and let $f \in C^{\lfloor q\rfloor}(\R^d, \R)$. Then $Y^0 \assign 1$,
	\[
		Y^k(s) \assign \nabla^{k} f(S(s)),\qquad k=1,\dots, \lfloor q \rfloor
	\]
	defines a controlled path in $\mc D^{\lfloor q \rfloor /q}_{\bb X}([0,T])$. Indeed, as we discussed above $\nabla^k f(S(s)) \in {\rm Sym}_k(\R^d)$ for all $k = 1,\dots, \lfloor q \rfloor$, and  by Taylor's formula we have for $\ell \in \{1, \dots, \lfloor q \rfloor\}$
	\ba
		Y^\ell(t) = \nabla^\ell f(S(t)) & = \sum_{k=\ell}^{\lfloor q \rfloor} \frac{1}{(k-\ell)!} \langle \nabla^{k} f(S(s)), (S(t) - S(s))^{\otimes (k-\ell)} \rangle + O(c(s,t)^{(\lfloor q \rfloor - \ell + 1)/q}) \\
		& = \sum_{k=\ell}^{\lfloor q \rfloor} \langle Y^k(s), \bb X_{s,t}^{k-\ell} \rangle + O(c(s,t)^{(\lfloor q \rfloor - \ell + 1)/q}).
	\ea
\end{example}

\begin{proposition}\label{prop:rp-int}
	Let $p \ge 1$, let $\bb X$ be a reduced rough path of finite $p$-variation and let $Y \in \mc D^{\lfloor p \rfloor /p}_{\bb X}([0,T])$. Then the rough path integral
	\[
		I_{\bb X}(Y)(t)=\int_0^t \langle Y(s), \dd \bb X(s)\rangle = \lim_{\substack{\pi \in \Pi([0,t])\\ |\pi| \to 0}} \sum_{[t_j,t_{j+1}] \in \pi} \sum_{k=1}^{\lfloor p \rfloor} \langle Y^k(t_j), \bb X^k_{t_j,t_{j+1}}\rangle,\qquad t \in [0,T],
	\]
	defines a function in $C([0,T],\R)$, and it is the unique function with $I_{\bb X}(Y)(0)=0$ for which there exists a control function $c$ with
	\[
		\Big|\int_s^t \langle Y(r), \dd \bb X(r)\rangle - \sum_{k=1}^{\lfloor p \rfloor} \langle Y^k(s), \bb X^k_{s,t}\rangle\Big| \lesssim c(s,t)^{\frac{\lfloor p \rfloor+1}{p}},\qquad (s,t) \in \Delta_T.
	\]
\end{proposition}

\begin{proof}
	This follows from  classical arguments (Theorem~4.3 in~\cite{Lyons2007}, see also~\cite{gubinelli2004}) once we show that for $0 \le s \le u \le t \le T$
	\[
		\sum_{k=1}^{\lfloor p \rfloor} \langle Y^k(s), \bb X^k_{s,t}\rangle - \sum_{k=1}^{\lfloor p \rfloor} \langle Y^k(s), \bb X^k_{s,u}\rangle - \sum_{k=1}^{\lfloor p \rfloor} \langle Y^k(u), \bb X^k_{u,t}\rangle = O(c(s,t)^{\frac{\lfloor p \rfloor+1}{p}}),
	\]
	where $c$ is a control function such that the estimates in Definition~\ref{def:reduced-rp} and in Definition~\ref{def:controlled} hold. But
	\ba
		\sum_{k=1}^{\lfloor p \rfloor} \langle Y^k(u), \bb X^k_{u,t}\rangle & = \sum_{k=1}^{\lfloor p \rfloor} \Big(\sum_{\ell=k}^{\lfloor p \rfloor} \langle Y^\ell(s), \bb X^{\ell-k}_{s,u} \otimes \bb X^k_{u,t}\rangle + O\big(c(s,u)^{\frac{\lfloor p \rfloor - k + 1}{p}} c(u,t)^{\frac{k}{p}}\big) \Big) \\
		& = \sum_{k=1}^{\lfloor p \rfloor} \sum_{\ell=1}^k \langle Y^k(s), \bb X^{k-\ell}_{s,u} \otimes \bb X^\ell_{u,t}\rangle + O\big(c(s,t)^{\frac{\lfloor p\rfloor +1}{p}}\big) \\
		& = \sum_{k=1}^{\lfloor p \rfloor} \langle Y^k(s), \mathcal{P}_k({\rm Sym}(\bb X_{s,u} \otimes \bb X_{u,t})) - \bb X^k_{s,u}\rangle + O\big(c(s,t)^{\frac{\lfloor p\rfloor +1}{p}}\big),
	\ea
	where in the last step we used that $Y^k(s)$ is symmetric. Therefore, the reduced Chen relation gives
	\ba
		&\sum_{k=1}^{\lfloor p \rfloor} \langle Y^k(s), \bb X^k_{s,t}\rangle - \sum_{k=1}^{\lfloor p \rfloor} \langle Y^k(s), \bb X^k_{s,u}\rangle - \sum_{k=1}^{\lfloor p \rfloor} \langle Y^k(u), \bb X^k_{u,t}\rangle \\
		&\hspace{50pt} = \sum_{k=1}^{\lfloor p \rfloor} \langle Y^k(s), \bb X^k_{s,t} - \bb X^k_{s,u} - \mathcal{P}_k({\rm Sym}(\bb X_{s,u} \otimes \bb X_{u,t})) + \bb X^k_{s,u}\rangle + O\big(c(s,t)^{\frac{\lfloor p\rfloor +1}{p}}\big) \\
		&\hspace{50pt} = \sum_{k=1}^{\lfloor p \rfloor} \langle Y^k(s), \bb X^k_{s,t} - \bb X^k_{s,t}\rangle + O\big(c(s,t)^{\frac{\lfloor p\rfloor +1}{p}}\big) = O\big(c(s,t)^{\frac{\lfloor p\rfloor +1}{p}}\big),
	\ea
	which concludes the proof.
\end{proof}

\begin{corollary}
	Let $p\in \N$ be an even integer and let $q, S,\bb X, f$ be as in Example~\ref{ex:controlled-f}. Then
	\[
		\int_0^t \langle\nabla f(S(s)), \dd \bb X(s)\rangle = \int_0^t \langle \nabla f(S(s)), \dd S(s) \rangle,\qquad t \in [0,T],
	\]
	where the left hand side denotes the rough path integral of Proposition~\ref{prop:rp-int} and the right hand side is the integral of Theorem~\ref{thm:follmer-ito-multidim}.
\end{corollary}

\begin{proof}
	It suffices to show that
	\[
		\int_0^t \langle\nabla f(S(s)), \dd \bb X(s)\rangle = f(S(t)) - f(S(0)) - \frac{1}{p!} \int_0^t \langle \nabla^p f(S(s)), \dd [S]^p(s) \rangle,
	\]
	and since
	\[
		\lim_{\substack{\pi \in \Pi([0,t])\\ |\pi| \to 0}} \sum_{[t_j,t_{j+1}] \in \pi} \langle \nabla^p f(S(t_j)),  [S]^p(t_{j+1}) - [S]^p(t_j)\rangle = \int_0^t \langle \nabla^p f(S(s)), \dd [S]^p(s) \rangle,
	\]
	this is equivalent to
	\[
		\lim_{\substack{\pi \in \Pi([0,t])\\ |\pi| \to 0}} \sum_{[t_j,t_{j+1}] \in \pi} \sum_{k=1}^{ p } \langle \nabla^k f(S(t_j)), \frac{1}{k!}(S(t_{j+1}) - S(t_j))^{\otimes k} \rangle = f(S(t)) - f(S(0)).
	\]
	The last identity can be shown by writing $f(S(t)) - f(S(0))$ as a telescoping sum and by performing a Taylor expansion up to order $p$ and controlling the remainder term as in the proof of Theorem~\ref{thm:follmer-ito-multidim}.
\end{proof}


\newpage
\appendix

\section*{Appendix: $p$-th variation for odd integer values of $p$}\label{app:odd}

\begin{lemma}
  Let $p > 1$ be an odd integer and let $\pi_n$ be the dyadic Lebesgue partition generated by $S \in C([0,T],\R)$. Assume that
  $\nu^n \assign \sum_{[t_j, t_{j + 1}] \in \pi_n} \delta (\cdot - t_j)
  | S (t_{j + 1}) - S (t_j) |^p$ converges weakly to a signed
  measure $\nu$ without atoms. Then we have for all $f \in C(\R,\R)$
  \[
  	\lim_{n \rightarrow \infty} \sum_{[t_j, t_{j + 1}] \in \pi_n} f (S (t_j)) (S (t_{j + 1} \wedge t) - S (t_j \wedge t))^p = 0, \qquad t \in [0,T].
  \]
\end{lemma}

\begin{proof}
  We can assume without loss of generality that $f$ has compact support, since
  the image of $S$ on $[0, T]$ is compact. Let $k \in
  \mathbb{Z}$ and note that whenever $S$ completes an upcrossing of $I^n_k = [k2^{-n},(k+1)2^{-n}]$ we have to add $f (k 2^{- n}) 2^{- n p}$ to the sum.
  On the other hand, if $S$ completes a downcrossing of $I^n_k$
  before $t$, then we have to add $- f ((k + 1) 2^{- n}) 2^{- n p}$ to the
  sum. Let $U_t(I^n_k)$ (resp. $D_t(I^n_k)$) denote the number of up- (resp. down-) crossings of $I^n_k$ by $S$ on $[0,t]$. Since $U_t(I^n_k)$ and $D_t(I^n_k)$ differ by at
  most 1, we get
  \begin{align*}
    &\left| \sum_{[t_j, t_{j + 1}] \in \pi_n : t_{j + 1} \leqslant t} f (S
    (t_j)) (S (t_{j + 1}) - S (t_j))^p \right|\\
    &\hspace{30pt} = \left| \sum_{k \in \mathbb{Z}} 2^{- n p} (f (k 2^{- n}) U_t (I^n_k) - f ((k + 1) 2^{- n}) D_t (I^n_k)) \right|\\
    &\hspace{30pt} \leqslant \left| \sum_{k \in \mathbb{Z}} 2^{- n p} f (k 2^{- n}) (U_t (I^n_k) - D_t (I^n_k)) \right| + \left| \sum_{k \in \mathbb{Z}} 2^{- n p} (f (k 2^{- n}) - f ((k + 1)
    2^{- n})) D_t (I^n_k) \right|\\
    &\hspace{30pt} \leqslant \sum_{k \in \mathbb{Z}} 2^{- n p} | f (k 2^{- n}) | + \sum_{k \in \mathbb{Z}} 2^{- n p} | (f (k 2^{- n}) - f ((k + 1) 2^{-
    n})) | N_t (I^n_k)\\
    &\hspace{30pt} \leqslant \sum_{k \in \mathbb{Z}} 2^{- n p} | f (k 2^{- n}) | +
    \omega_f( 2^{- n}) \sum_{k \in \mathbb{Z}} 2^{- n p} N_t (I^n_k),
  \end{align*}
  where we wrote $N_t(I^n_k) = U_t(I^n_k) + D_t(I^n_k)$ for the total number of interval crossings and where $\omega_f$ is the modulus of continuity of $f$, i.e. $\lim_{n \to \infty} \omega_f(2^{-n}) = 0$.
  By assumption,
  \[ \lim_{n \rightarrow \infty} \sum_{k \in \mathbb{Z}} 2^{- n p} N_t (I^n_k) = \nu ([0, t]) \in \mathbb{R}, \]
  and since $f (k 2^{- n}) \neq 0$ for at most $O (2^n)$ values of $k$ and
   $p > 1$ the claim follows.
\end{proof}

\end{document}